\documentclass{amsart}
\usepackage{setspace}

\usepackage{amsthm}
\usepackage{a4}
\usepackage{amsfonts}
\usepackage{amsmath}
\usepackage{amssymb}
\usepackage{mathrsfs}
\usepackage{mathtools}

\usepackage{enumerate}

\usepackage{color}

\parskip\medskipamount
\DeclareFontFamily{OT1}{manual}{}
\DeclareFontShape{OT1}{manual}{m}{n}{ <10> manfnt }{}

\makeatletter
\def\blfootnote{\xdef\@thefnmark{}\@footnotetext}
\makeatother

\newcommand*{\<}{\langle}
\renewcommand*{\>}{\rangle}
\newcommand*{\x}{\times}

\newcommand{\leqs}{\leqslant}
\newcommand{\geqs}{\geqslant}

\newtheorem{lemma}{Lemma}[section]
\newtheorem{theorem}[lemma]{Theorem}
\newtheorem{prop}[lemma]{Proposition}
\newtheorem{cor}[lemma]{Corollary}

\newtheorem{question}[lemma]{Question}

\theoremstyle{definition}

\theoremstyle{remark}
\newtheorem{remark}[lemma]{Remark}

\theoremstyle{definition}
\newtheorem{example}[lemma]{Example}

\title{Group and round quadratic forms}
\author{James O'Shea}
\address{James O'Shea,\newline Department of Mathematics and Statistics,\newline National University of Ireland Maynooth.\newline E-mail: james.oshea@nuim.ie}
\date{}
\begin{document}

\maketitle

%


\begin{abstract}\noindent We offer some elementary characterisations of group and round quadratic forms. These characterisations are applied to establish new (and recover existing) characterisations of Pfister forms. We establish ``going-up'' results for group and anisotropic round forms with respect to iterated Laurent series fields, which contrast with the established results with respect to rational function field extensions. For forms of two-power dimension, we determine when there exists a field extension over which the form becomes an anisotropic group form that is not round.
\end{abstract}


\blfootnote{James O'Shea,\\ Department of Mathematics and Statistics, National University of Ireland Maynooth, Ireland.\\ \emph{Previous affiliation}: National University of Ireland Galway.\\ \emph{E-mail}: james.oshea@nuim.ie}





\section{Introduction} 






A quadratic form is round if its value set coincides with the multiplicative group of similarity factors associated with the form. Thus, round forms constitute a prominent subclass of group forms, forms whose value sets are multiplicative groups. As roundness is one of the fundamental properties of Pfister forms, the class of forms that occupy a central role in quadratic form theory, it is unsurprising that this notion has had a number of important consequences. However, while the structure and behaviour of round forms has received extensive treatment in the literature, this class of forms is still not fully understood and, as suggested in \cite{LAM}, merits further study. The broader class of group forms is comparatively little understood.


Our opening results, which are invoked throughout this article, record elementary characterisations of the classes of group and round forms (see Proposition~\ref{group1} to Corollary~\ref{roundchar}). In Section 2, we apply these results to obtain new characterisations of Pfister forms (see Theorem~\ref{transgroup}), in addition to re-proving established ones (see Corollary~\ref{Pfisterreproof}), and to extend Elman's classification of odd-dimensional round forms in accordance with our broader definition of roundness.

The group and round properties of a form are intrinsically linked to its base field of definition, and thus are sensitive to scalar extension. In \cite{A}, Alpers remarks that while general ``going-down'' results exist with respect to roundness, with round forms over odd-degree extensions being round over their base fields for example, no general results are known in the ``going-up'' direction. We establish such results for group and anisotropic round forms with respect to iterated Laurent series fields (see Corollary~\ref{laurentcor}), highlighting an interesting divergence in the behaviour of forms under extension to iterated Laurent series fields as opposed to rational function fields (see Remark~\ref{divergence}). 



In \cite{HJ1} and \cite{HJ2}, Hsia and Johnson studied the problem of distinguishing between anisotropic group and round forms over local and global fields. In this spirit, we consider the following general question:



\begin{question}\label{groupnotround} For $q$ an anisotropic form over $F$, does there exist an extension $K/F$ such that $q_K$ is an anisotropic group form that is not round?
\end{question}

Our characterisation of group forms allows for the construction of a generic field extension over which a form becomes an anisotropic group form. Thus, the adoption of Merkurjev's method of passing to iterated field extensions obtained by composing function fields of quadratic forms represents the natural approach to addressing the above question. While highlighting an obstruction to resolving Question~\ref{groupnotround} in general (see Proposition~\ref{isogroup}), we can employ this method to good effect in certain situations. In particular, Theorem~\ref{gnr} represents a complete answer to Question~\ref{groupnotround} with respect to forms of two-power dimension.

We let $F$ denote a field of characteristic different from two, and recall that every non-degenerate quadratic form on a vector space over $F$ can be diagonalised. We write $\<a_1,\ldots,a_n\>$ to denote the ($n$-dimensional) quadratic form on an $n$-dimensional $F$-vector space defined by $a_1,\ldots ,a_n\in F^{\x}$. We use the term ``form'' to refer to a non-degenerate quadratic form of positive dimension. If $p$ and $q$ are forms over $F$, we denote by $p\perp q$ their orthogonal sum and by $p\otimes q$ their tensor product. We use $aq$ to denote $\<a\>\otimes q$ for $a\in F^{\x}$. We write $p\simeq q$ to indicate that $p$ and $q$ are isometric, and say that $p$ and $q$ are \emph{similar} if $p\simeq aq$ for some $a\in F^{\x}$. A form $p$ is a \emph{subform of $q$} if $q\simeq p\perp r$ for some form $r$, in which case we write $p\subseteq q$. For $q$ a form over $F$ and $K/F$ a field extension, we will often employ the notation $q_K$ when viewing $q$ as a form over $K$ via the canonical embedding. A form $q$ represents $a\in F$ if there exists a vector $v$ such that $q(v)=a$. We denote by $D_F(q)$ the set of values in $F^{\x}$ represented by $q$. A form over $F$ is \emph{isotropic} if it represents zero non-trivially, and \emph{anisotropic} otherwise. Every form $q$ has a decomposition $q\simeq q_{\mathrm{an}}\perp i(q)\x\<1,-1\>$ where the anisotropic form $q_{\mathrm{an}}$ and the non-negative integer $i(q)$ are uniquely determined. If a form $q$ is isotropic over $F$, then $D_F(q)=F^{\x}$, as $\left(\frac {a+1} 2\right)^2-\left(\frac {a-1} 2\right)^2=a$ for all $a\in F^{\x}$. A form $q$ is \emph{hyperbolic} if $i(q)=\frac 1 2 \dim q$. 


%






A form $q$ is a \emph{group form} over $F$ if $D_F(q)$ is a subgroup of $F^{\x}$. The similarity factors of $q$ generate the group $G_F(q)=\{a\in F^{\x}\mid aq\simeq q\}$. Equivalently, $G_F(q)=\{a\in F^{\times}\mid \< 1,-a\>\otimes q\text{ is hyperbolic}\}$. A group form $q$ over $F$ is said to be \emph{round} if $D_F(q)=G_F(q)$. Equivalently, a form $q$ is round over $F$ if $D_F(q)\subseteq G_F(q)$, as if $a\in D_F(q)\subseteq G_F(q)$, then $aq\simeq q$, whereby $1\in D_F(q)$ and thus $G_F(q)\subseteq D_F(q)$. We use $H_F(q)$ to denote the set of products of two elements of $D_F(q)$. Per Lemma~\ref{R}, we have that $H_F(q)=\{a\in F^{\times}\mid \< 1,-a\>\otimes q\text{ is isotropic}\}.$ For $n\in\mathbb{N}$, an \emph{$n$-fold Pfister form} over $F$ is a form isometric to $\<1,a_1\>\otimes\ldots\otimes\<1,a_n\>$ for some $a_1,\ldots ,a_n\in F^{\times}$ (the form $\< 1\>$ is the $0$-fold Pfister form). Isotropic Pfister forms are hyperbolic \cite[Theorem X.1.7]{LAM}. Pfister forms are round (see \cite[Theorem X.1.8]{LAM}). A form $\tau$ is a Pfister neighbour if $\tau$ is similar to a subform of a Pfister form $\pi$ and $\dim{\tau}>\frac 1 2 \dim\pi$.

We recall that every non-zero square class in $F(\!(x)\!)$, the Laurent series field in the variable $x$ over $F$, can be represented by $a$ or $ax$ for some $a\in F^{\x}$, whereby every form over $F(\!(x)\!)$ can be written as $p\perp xq$ for $p$ and $q$ forms over $F$. We will often invoke the following folkloric result regarding the isotropy over Laurent series fields.

\begin{lemma}\label{Hlemma} Let $p$ and $q$ be forms over $F$. Considering $p\perp x q$ as a form over $F(\!(x)\!)$, we have that $i(p\perp x q)=i(p)+ i(q)$.
\end{lemma}




For a form $q$ over $F$ with $\dim q=n\geqs 2$ and $q\not\simeq\<1,-1\>$, the \emph{function field $F(q)$ of $q$} is the quotient field of the integral domain $F[X_1,\ldots ,X_n]/(q(X_1,\ldots ,X_n))$ (this is the function field of the affine quadric $q(X)=0$ over $F$). To avoid case distinctions, we set $F(q)=F$ if $\dim q= 1$ or $q\simeq\<1,-1\>$. For $q$ a form over $F$, we note the inclusion $F(\!(x)\!)(q)\subseteq F(q)(\!(x)\!)$, which we will apply in combination with Lemma~\ref{Hlemma}. For all forms $p$ over $F$ and all extensions $K/F$ such that $q_K$ is isotropic, we have that $i(p_{F(q)})\leqs i(p_K)$ in accordance with Knebusch's specialisation results \cite[Proposition 3.1 and Theorem 3.3]{Kn1}. In particular, letting $i_1(q)$ denote $i(q_{F(q)})$, we have that $i_1(q)\leqs i(q_K)$ for all extensions $K/F$ such that $q_K$ is isotropic. Invoking the Cassels-Pfister Subform Theorem \cite[Theorem X.$4.5$]{LAM} of Wadsworth \cite[Theorem 2]{W} and Knebusch \cite[Lemma 4.5]{Kn1}, for $p$ and $q$ anisotropic forms over $F$ of dimension at least two such that $p_{F(q)}$ is hyperbolic, one has that $aq\subseteq bp$ for all $a\in D_F(q)$ and $b\in D_F(p)$. For $q$ an anisotropic form over $F$ of dimension at least two, it is known that $q_{F(q)}$ is hyperbolic if and only if $q$ is similar to a Pfister form over $F$ by \cite[Theorem X.$4.14$]{LAM}, a result of Wadsworth \cite[Theorem 5]{W} and Knebusch \cite[Theorem 5.8]{Kn1}. We will regularly invoke \cite[Theorem 1]{H}, Hoffmann's Separation Theorem, which we recall below.



\begin{theorem}\label{H95} Let $p$ and $q$ be forms over $F$ such that $p$ is anisotropic. If $\dim p\leqs 2^n<\dim q$ for some integer $n\geqs 0$, then $p_{F(q)}$ is anisotropic.\end{theorem}

In accordance with the above theorem and \cite[Exercise I.$16$]{LAM}, for $q$ an anisotropic form over $F$, we note that $\dim q$ and $\dim q-i_1(q)$ belong to an interval of the form $[2^n,2^{n+1}]$ for some $n\in\mathbb{N}\cup\{ 0\}$. We will also invoke the following isotropy criterion of Karpenko and Merkurjev \cite[Theorem 4.1]{KM}.




\begin{theorem}\label{km} For $p$ and $q$ anisotropic forms over $F$ such that $p_{F(q)}$ is isotropic, 
\begin{enumerate}[$(i)$]\item $\dim p - i_1(p)\geqs\dim q - i_1(q)$;
\smallskip
\item $\dim p - i_1(p)=\dim q - i_1(q)$ if and only if $q_{F(p)}$ is isotropic.\end{enumerate}\end{theorem}

We refer to works by Vishik \cite{Vishik} and Scully \cite{S1} for recent results in the spirit of Theorem~\ref{H95} and Theorem~\ref{km}.









\section{Characterisations of group, round and Pfister forms}

As above, the forms we consider are non-degenerate and of positive dimension over fields of characteristic different from two. In accordance with the associated definitions, we begin our study of the properties of a form $q$ over $F$ being group or round by considering the value set $D_F(q)$, the group of similarity factors $G_F(q)$, and the set of products of two elements of $D_F(q)$, usually denoted by $D_F(q)D_F(q)$.

The following result appears in Roussey's thesis (see \cite[Lemme 2.5.4]{R}).

\begin{lemma}\label{R} For $p$ and $q$ forms over $F$, we have that $$D_F(p)D_F(q)=\{a\in F^{\x}\mid p\perp -a q \text{ is isotropic}\}.$$
\end{lemma}

\begin{proof} The statement clearly holds if either $p$ or $q$ is isotropic. Thus, assuming that $p$ and $q$ are anisotropic, we have that $p\perp -a q$ is isotropic if and only if there exist non-zero vectors $v$ and $w$ such that $p(v)-aq(w)=0$. Thus, $p(v)=aq(w)\neq 0$, whereby $a=p(v)(q(w))^{-1}$. Hence, $$a=p(v)\left(\frac 1 {q(w)}\right)=p(v)q\left(\frac w {q(w)}\right)\in D_F(p)D_F(q).$$ As $1\in D_F(dq)$ for $d\in D_F(q)$, we have that $p\perp -cdq$ is isotropic for $c\in D_F(p)$.
\end{proof}



We let $H_F(q)=\{a\in F^{\times}\mid \< 1,-a\>\otimes q\text{ is isotropic}\}$, whereby $H_F(q)=D_F(q)D_F(q)$ in accordance with Lemma~\ref{R}. As with $D_F(q)$ and $G_F(q)$, we may restrict our attention to the square classes contained in $H_F(q)$, since $H_F(q)=(F^{\x})^2H_F(q)$. Clearly we have that $(F^{\x})^2\subseteq G_F(q)\subseteq H_F(q)$ for all forms $q$ over $F$. Moreover, if $1\in D_F(q)$, then we have that $$(F^{\x})^2\subseteq G_F(q)\subseteq D_F(q)\subseteq H_F(q).$$

If $q$ is isotropic over $F$, then $q$ is a group form over $F$, with $D_F(q)=F^{\x}=H_F(q)$ in this case. Our opening result records that Lemma~\ref{R} may be applied to characterise group forms.







\begin{prop}\label{group1} A form $q$ is a group form over $F$ if and only if $H_F(q)\subseteq D_F(q)$.\end{prop}

\begin{proof} The non-empty set $D_F(q)$ is clearly associative. Letting $a\in D_F(q)$, there exists a non-zero vector $v$ such that $q(v)=a$, whereby $q(a^{-1}v)=(a^{-1})^{2}q(v)=a^{-1}\in D_F(q)$. Thus, $q$ is a group form over $F$ if and only if $D_F(q)D_F(q)\subseteq D_F(q)$, whereby the result follows by invoking Lemma~\ref{R}.
\end{proof}


As group forms represent $1$, we thus obtain the following corollary. 

\begin{cor}\label{group2} A form $q$ is a group form over $F$ if and only if $H_F(q) =D_F(q)$.\end{cor}




In accordance with our definition of roundness, if $q$ is isotropic over $F$, then $q$ is round over $F$ if and only if $q$ is hyperbolic or the non-zero form $q_{\text{an}}$ is such that $D_F\left(q_{\text{an}}\right)=F^{\x}=G_F\left(q_{\text{an}}\right)$. This observation follows from the fact that $G_F(q)=G_F(q_{\text{an}})$, in accordance with Witt Cancellation and the fact that $D_F\left(q\right)=F^{\x}$ for $q$ isotropic over $F$. Per \cite[Example X.$1.15(5)$]{LAM}, the form $q\simeq\< 1,-1,1,1\>$ over $\mathbb{F}_3$ is an example of an isotropic round form that is not hyperbolic.









%

%


\begin{cor}\label{genroundchar} A form $q$ is round over $F$ if and only if $1\in D_F(q)$ and $H_F(q)\subseteq G_F(q)$.\end{cor}

\begin{proof} If $q$ is round over $F$, then $1\in D_F(q)=G_F(q)$ and $q$ is a group form over $F$. Invoking Proposition~\ref{group1}, it follows that $H_F(q)\subseteq D_F(q)$, whereby $H_F(q)\subseteq G_F(q)$.

Conversely, as $1\in D_F(q)$, we recall that $G_F(q)\subseteq D_F(q)\subseteq H_F(q)$, whereby the equality $D_F(q)=G_F(q)$ follows from the assumption that $H_F(q)\subseteq G_F(q)$.
\end{proof}



We note that, for $q$ a round form over $F$, the inclusion $H_F(q)\subseteq G_F(q)$ can also be derived from \cite[Proposition 1 and Theorem 1]{WS}.

Addressing the question of distinguishing between the classes of group and round forms over a given field, it is reasonable to restrict one's consideration to those forms that represent one, whereby the preceding characterisation may be simplified.


\begin{cor}\label{roundchar} Let $q$ be a form such that $1\in D_F(q)$. The following are equivalent:
\begin{enumerate}[$(i)$]
\item $q$ is round over $F$,
\item $H_F(q)\subseteq G_F(q)$,
\item $q\otimes \rho$ is anisotropic or hyperbolic for every $1$-fold Pfister form $\rho$ over $F$,
\item $q\otimes \beta$ is anisotropic or hyperbolic for every $2$-dimensional form $\beta$ over $F$,
\item $q\otimes \pi$ is anisotropic or hyperbolic for every $n$-fold Pfister form $\pi$ over $F$, $n\in\mathbb{N}$.
\end{enumerate}
\end{cor}


\begin{proof} Statements $(i)$, $(ii)$ and $(iii)$ are equivalent by Corollary~\ref{genroundchar}. Statements $(iii)$ and $(iv)$ are equivalent, as scaling does not affect isotropy. Statement $(v)$ clearly implies $(iii)$. Assuming $(i)$, it follows that $q\otimes \pi$ is round for every Pfister form $\pi$ over $F$, by Witt's Round Form Theorem \cite[Theorem X.$1.14$]{LAM}. By repeatedly invoking Statement $(iii)$, we see that $(v)$ follows.
\end{proof}




In the context of forms that represent one, we remark that scalar multiples of Pfister forms are Pfister forms.

\begin{lemma}\label{simone} A form $q$ over $F$ is a Pfister form if and only if $q$ is similar to a Pfister form and $1\in D_F(q)$.
\end{lemma}

\begin{proof} To establish the right-to-left implication, we let $q\simeq a\pi$ for $a\in F^{\x}$ and $\pi$ a Pfister form over $F$. As $1\in D_F(q)$, it follows that $a\in D_F(\pi)$, whereby $q\simeq \pi$ as $\pi$ is round.
\end{proof}

We can apply the above characterisations of round and group forms to obtain a new characterisation of Pfister forms.

\begin{theorem}\label{transgroup} Let $q$ be an anisotropic form. The following are equivalent: \begin{enumerate}[$(i)$]
\item $q$ is a Pfister form over $F$, 
\smallskip
\item $q$ is a round form over $K=F(\!(x)\!)(q\otimes \< 1,-x\>)$,
\smallskip
\item $q$ is a group form over $K=F(\!(x)\!)(q\otimes \< 1,-x\>)$.
\end{enumerate}
\end{theorem}


\begin{proof} As Pfister forms are round, and round forms are group, it suffices to prove that $(iii)$ implies $(i)$.

The field $K$ is the function field of $q\otimes \< 1,-x\>$ over $F(\!(x)\!)$, which is an anisotropic form in accordance with Lemma~\ref{Hlemma}. We will first show that $D_K(q)\cap F^{\x}=D_F(q)$. Let $a\in F^{\x}$ be such that $q\perp \< -a\>$ is anisotropic over $F$ and suppose, for the sake of contradiction, that $q\perp \< -a\>$ is isotropic over $K$. Invoking Theorem~\ref{km}~$(i)$, we have that $$\dim(q\perp\< -a\>)-i_1(q\perp\< -a\>)\geqs \dim (q\otimes \< 1,-x\>)-i_1(q\otimes \< 1,-x\>).$$ In accordance with Theorem~\ref{H95} and \cite[Exercise I.$16$]{LAM}, there exists $n\in\mathbb{N}$ such that $$\dim(q\perp\< -a\>)-i_1(q\perp\< -a\>)= \dim (q\otimes \< 1,-x\>)-i_1(q\otimes \< 1,-x\>)=2^n,$$ whereby $\dim q=2^n$. Hence, $q\otimes \< 1,-x\>$ is isotropic over $F(\!(x)\!)(q\perp\< -a\>)$, in accordance with Theorem~\ref{km}~$(ii)$. Invoking \cite[Lemma 5.4 $(3)$]{Ilow}, it thus follows that $q$ is isotropic over $F(q\perp\< -a\>)$. However, as $\dim q=2^n$, this contradicts Theorem~\ref{H95}, thereby establishing the claim. 


We have that $1\in D_K(q)$ by assumption, whereby $1\in D_F(q)$ by the statement proven above. As $x\in H_K(q)$ by construction, it follows that $x\in D_K(q)$ in accordance with Proposition~\ref{group1}, whereby the form $q\perp\< -x\>$ becomes isotropic over $F(\!(x)\!)(q\otimes \< 1,-x\>)$. Arguing as above, it follows that, for some $n\in\mathbb{N}$, we have that $$\dim(q\perp\< -x\>)-i_1(q\perp\< -x\>)= \dim (q\otimes \< 1,-x\>)-i_1(q\otimes \< 1,-x\>)=2^n,$$ and that $\dim q=2^n$. Hence, $i_1(q\otimes \< 1,-x\>)=\dim q$, in accordance with this equality, whereby $q\otimes \< 1,-x\>$ becomes hyperbolic over $F(\!(x)\!)(q\otimes \< 1,-x\>)$. Invoking \cite[Proposition 3.2]{OS2}, it follows that $q$ is hyperbolic over $F(q)$. Thus, $q$ is similar to a Pfister form over $F$ by \cite[Theorem X.$4.14$]{LAM}. As $1\in D_F(q)$, the result follows by invoking Lemma~\ref{simone}.
\end{proof}




We can invoke the above result to re-prove the following characterisations of Pfister forms due to Pfister (see \cite[Satz 5, Theorem 2]{P}, \cite[Theorem $4.4$, p.153]{S} or \cite[Theorem $23.2$]{EKM}).

\begin{cor}\label{Pfisterreproof} Let $q$ be an anisotropic form over $F$. The following are equivalent: \begin{enumerate}[$(i)$]
\item $q$ is a Pfister form over $F$,
\smallskip
\item $q$ is round over $K$ for every extension $K/F$,
\smallskip
\item $q$ is group over $K$ for every extension $K/F$.
\end{enumerate}
\end{cor}

Similarly, we obtain the following corollary of Theorem~\ref{transgroup}.

\begin{cor}\label{Pfisterreformulation} For $q$ an anisotropic form over $F$, let $K=F(\!(x)\!)(q\otimes \< 1,-x\>)$.
\begin{enumerate}[$(i)$]
\item $q$ is group over every extension of $F$ if and only if $q$ is group over $K$,
\smallskip
\item $q$ is round over every extension of $F$ if and only if $q$ is round over $K$.
\end{enumerate}
\end{cor}


\begin{remark}\label{Pfistergoingdown} Per \cite[Theorem $4.4$, p.153]{S}, Pfister established that, for $q$ an anisotropic form of dimension $n$, the three statements in Corollary~\ref{Pfisterreproof} are equivalent to each of the following statements \begin{enumerate}[$(i)$]
\item $q\left(x_1,\ldots ,x_n\right)\in G_{K}(q)\text{ for }K=F\left(x_1,\ldots ,x_{n}\right)$,
\smallskip
\item $q\left(x_1,\ldots ,x_n\right)q\left(x_{n+1},\ldots ,x_{2n}\right)\in D_{K}(q)\text{ for }K=F\left(x_1,\ldots ,x_{2n}\right)$. 
\end{enumerate}
Thus, for $q$ an anisotropic form of dimension $n$, it follows that $q$ is a Pfister form over $F$ if and only if $q$ is a round form over $F\left(x_1,\ldots ,x_{n}\right)$, and that $q$ is a Pfister form over $F$ if and only if group form over $F\left(x_1,\ldots ,x_{2n}\right)$.
\end{remark}








In a similar spirit to the preceding results, we offer the following characterisation of scalar multiples of Pfister forms.

\begin{prop}\label{isoext} Let $q$ be an anisotropic form over $F$. The following are equivalent:
\begin{enumerate}[$(i)$]
\item $q$ is similar to a Pfister form over $F$,
\smallskip
\item $q_K$ is round for all extensions $K/F$ such that $1\in D_K(q)$,
\smallskip
\item $q_K$ is round for all extensions $K/F$ such that $q_K$ is isotropic.
\end{enumerate}
\end{prop}

\begin{proof} Assuming $(i)$, Lemma~\ref{simone} implies that $q_K$ is a Pfister form for $K/F$ such that $1\in D_K(q)$, whereby $(ii)$ follows. As $(ii)$ clearly implies $(iii)$, it suffices to show that $(iii)$ implies $(i)$.


Letting $K = F(q)(\!(x)\!)$, we have that $D_K(q)=K^{\x}=G_K(q)$. As $x\in G_K(q)$, the form $q\otimes\< 1,-x\>$ is hyperbolic over $F(q)(\!(x)\!)$. Invoking Lemma~\ref{Hlemma}, it follows that $q$ is hyperbolic over $F(q)$. Thus, $q$ is similar to a Pfister form over $F$ by \cite[Theorem X.$4.14$]{LAM}.
\end{proof}

We conclude this section by characterising the odd-dimensional round forms (see \cite{E} for Elman's characterisation of odd-dimensional round forms in the situation where isotropic round forms are defined to be hyperbolic).

\begin{prop}\label{rounddim} Let $q$ be a form over $F$. \begin{enumerate}[$(i)$]
\item If $D_F(q)=(F^{\x})^2$, then $q$ is round over $F$.
\smallskip
\item If $H_F(q)\neq (F^{\x})^2$ and $q$ is round over $F$, then $q$ is even-dimensional.
\end{enumerate}
\end{prop}


\begin{proof} $(i)$ If $D_F(q)=(F^{\x})^2$, then $D_F(q)\subseteq G_F(q)$, whereby $q$ is round over $F$.

$(ii)$ Let $a\in H_F(q)\setminus (F^{\x})^2$. As $q$ is round over $F$, we have that $H_F(q)\subseteq G_F(q)$ by Corollary~\ref{genroundchar}, whereby $q\perp -aq$ is hyperbolic over $F$. As a comparison of determinants yields the contradiction that $a\in (F^{\x})^2$ for $q$ odd-dimensional, the result follows.
\end{proof}



%




Adapting Elman's proof of \cite[Lemma]{E}, we obtain the following result as a corollary of Proposition~\ref{rounddim}.

\begin{cor}\label{oddround} Let $q$ be an odd-dimensional form over $F$. If $q$ is round, then $q\simeq (2r+1)\x\< 1\>$ for some $r\in \mathbb{N}\cup\{ 0\}$. Moreover, the following are equivalent:
\begin{enumerate}[$(i)$]
\item $(2k+1)\x\< 1\>$ is round over $F$ for some $k\in\mathbb{N}$,
\smallskip
\item $(2n+1)\x\< 1\>$ is round over $F$ for every $n\in\mathbb{N}\cup\{ 0\}$,
\smallskip
\item $F$ is Pythagorean.
\end{enumerate}
\end{cor}




\begin{cor}\label{oddisoround}
Let $q$ be an odd-dimensional isotropic form over $F$. Then $q$ is round over $F$ if and only if $F$ is quadratically closed.
\end{cor}

\begin{proof} If $q$ is round, then $D_F(q)=(F^{\x})^2$ by Proposition~\ref{rounddim}. As $q$ is isotropic, it follows that $D_F(q)=F^{\x}$, whereby $F$ is quadratically closed.


If $F$ is quadratically closed, then $q_{\text{an}}\simeq\< 1\>$, whereby $q$ is round.
\end{proof}




\begin{cor}\label{oddroundcor} If $q$ is an odd-dimensional anisotropic round form over $F$, then $q\otimes \beta$ is anisotropic over $F$ for every anisotropic $2$-dimensional form $\beta$ over $F$. \end{cor}

\begin{proof} Let $\beta\simeq b\< 1,-a\>$ be anisotropic over $F$ for $a,b\in F^{\x}$. Suppose, for the sake of contradiction, that $q\otimes \beta$ is isotropic over $F$. Hence, $q\otimes \< 1,-a\>$ is hyperbolic over $F$ by Corollary~\ref{genroundchar}. By repeatedly invoking \cite[Proposition 2.2]{EL2}, it follows that there exist binary forms $\beta_1,\ldots ,\beta_n$ over $F$ such that $\beta_i\otimes \< 1,-a\>$ is hyperbolic over $F$ for $i=1,\ldots ,n$ and such that $q\simeq\beta_1\perp\ldots\perp \beta_n$, whereby $q$ is even-dimensional, a contradiction.
\end{proof}

\begin{remark}\label{almosttrivial} The preceding result can also be derived from the fact that $F$ is Pythagorean and real, whereby its Witt ring is torsion free (see \cite[Theorem $23.2$]{EKM}).
\end{remark}


\section{Group and round forms over field extensions}


In \cite{A}, Alpers considers roundness with respect to algebraic extensions, establishing ``going-down'' and ``going-up'' results in certain situations. In particular, he remarks that a general going-down result holds for odd-degree extensions by Springer's theorem \cite{Sp} (see \cite[Theorem VII.$2.7$]{LAM}). We generalise this remark below.

\begin{prop}\label{Agen} Let $q$ be a form over $F$ and let $K$ be an extension of $F$.\begin{enumerate}[$(i)$]
\item Suppose that $D_K(q)\cap F^{\x}\subseteq D_F(q)$. If $q_K$ is a group form, then $q$ is a group form over $F$. 
\smallskip
\item Suppose that every anisotropic form over $F$ of dimension at most $\dim q +1$ is anisotropic over $K$. If $q_K$ is a round form, then $q$ is a round form over $F$. 
\end{enumerate}
\end{prop}

\begin{proof} $(i)$ As $D_K(q)\cap F^{\x}=D_F(q)$ follows from the assumption, if $D_K(q)$ is a group it readily follows that $D_F(q)$ is a group.

$(ii)$ If $q$ is anisotropic over $F$, the assumption on $K$ readily implies that $D_K(q)\cap F^{\x}=D_F(q)$ and $G_K(q)\cap F^{\x}=G_F(q)$, whereby the result follows. Applying this argument to $q_{\text{an}}$ in the case where $q$ is isotropic over $F$, the result follows.
\end{proof}


Thus, as a consequence of the above, group and round forms satisfy going-down results with respect to purely-transcendental extensions. Per Remark~\ref{Pfistergoingdown}, going-up results do not hold for group or round forms with respect to rational function fields. However, we do have the following result with respect to Laurent series fields:













%


\begin{prop}\label{laurent} Let $q$ be a form over $F$ and let $K=F(\!(x)\!)$.  
\begin{enumerate}[$(i)$]
\item $q$ is a group form over $F$ if and only if $q$ is a group form over $K$.
\smallskip
\item If $q$ is anisotropic, then $q$ is round over $F$ if and only if $q$ is round over $K$.
\end{enumerate}
\end{prop}

\begin{proof} We remark that anisotropic forms over $F$ are anisotropic over $K$.

$(i)$ As isotropic forms are trivially group, we may assume, without loss of generality, that $q$ is anisotropic over $F$. We consider the set $H_K(q)$, recalling that every non-zero square class in $K$ can be represented by $a$ or $ax$ for some $a\in F^{\x}$. Invoking Lemma~\ref{Hlemma}, it is apparent that $q\perp -axq$ is anisotropic over $K$ for $a\in F^{\x}$. For $a\in F^{\x}$ such that $q\perp -aq$ is isotropic over $K$, it follows that $q\perp -aq$ is isotropic over $F$, whereby $q\perp\< -a\>$ is isotropic over $F$ by Corollary~\ref{group1}. Thus, as $q\perp\< -a\>$ is isotropic over $K$, it follows from Corollary~\ref{group1} that $q$ is a group form over $K$.

For the converse statement, we may invoke Proposition~\ref{Agen}~$(i)$.

$(ii)$ As $1\in D_F(q)$ if and only if $1\in D_K(q)$, we may argue as in the preceding proof of $(i)$, with respect to Corollary~\ref{genroundchar} as opposed to Corollary~\ref{group1}, to establish the ``only if'' statement. The ``if'' statement can be established by invoking Proposition~\ref{Agen}~$(ii)$.
\end{proof}



\begin{remark}\label{anisround} We note the necessity of the restriction to anisotropic round forms in Statement $(ii)$ of the above result. If $q$ is isotropic and round over $K=F(\!(x)\!)$, then it readily follows that $q$ is isotropic and round over $F$, but the converse does not hold in general. In particular, the isotropic form $q\simeq\< 1,-1,1,1\>$ is round over $\mathbb{F}_3$ but it is not round over $\mathbb{F}_3(\!(x)\!)$, as $x\notin D_K\left(\left( q_K\right)_{\text{an}}\right)$ in accordance with Lemma~\ref{Hlemma}.
\end{remark}

Iterating the above, we obtain the following result.

\begin{cor}\label{laurentcor} Let $q$ be a form over $F$ and let $K=F(\!(x_1)\!)\ldots (\!(x_n)\!)$ for $n\in\mathbb{N}$.
\begin{enumerate}[$(i)$]
\item $q$ is a group form over $F$ if and only if $q$ is a group form over $K$.
\smallskip
\item If $q$ is anisotropic, then $q$ is round over $F$ if and only if $q$ is round over $K$.
\end{enumerate}
\end{cor}

%


\begin{remark}\label{divergence} We note that the above result demonstrates a divergence in the behaviour, with respect to the properties of being group or round, of forms over $F$ extended to iterated Laurent series fields as opposed to rational function fields. Corollary~\ref{laurentcor} contrasts with Pfister's result, per Remark~\ref{Pfistergoingdown}, that an anisotropic form of dimension $n$ over $F$ is a round form over $F\left(x_1,\ldots ,x_{n}\right)$ if and only if it is a group form over $F\left(x_1,\ldots ,x_{2n}\right)$ if and only if it is a Pfister form over $F$. 
\end{remark}








Motivated by the problem of distinguishing between anisotropic group and round forms, as studied over particular fields in \cite{HJ1} and \cite{HJ2}, the rest of this section is devoted to addressing Question~\ref{groupnotround}.







In accordance with Proposition~\ref{roundchar}, if $q$ is an anisotropic group form over $F$, one can resolve Question~\ref{groupnotround} by determining whether there exists a Pfister form $\pi$ over $F$ such that $q\otimes \pi$ is isotropic but not hyperbolic. If such a form $q$ is odd-dimensional, Question~\ref{groupnotround} further reduces to the problem of determining whether $H_F(q)\setminus (F^{\x})^2$ is empty, in accordance with Proposition~\ref{rounddim}.






The natural approach towards answering Question~\ref{groupnotround} in the case where $q$ is not a group form over $F$ is to consider its extension to the generic extension $K/F$ such that $q_K$ is a group form. However, one encounters the following obstruction:

\begin{prop}\label{isogroup} Let $q$ be an anisotropic form over $F$. If there exists $a\in H_F(q)$ such that $i_1(q\perp\< -a\>)>1$, then there does not exist an extension $K/F$ such that $q_K$ is an anisotropic group form.
\end{prop}

\begin{proof} Let $K/F$ be an extension such that $q_K$ is a group form. Since $a\in H_F(q)\subseteq H_K(q)$, it follows that $a\in D_K(q)$ by Corollary~\ref{group2}, whereby $q\perp\< -a\>$ is isotropic over $K$. Since $i_1(q\perp\< -a\>)>1$, it follows that $i((q\perp\< -a\>)_K)>1$, whereby $q_K$ is isotropic (see \cite[Exercise I.$16$]{LAM}).
\end{proof}


The following example illustrates that, provided that $\dim q\neq 2^n$ for $n\in\mathbb{N}\cup \{ 0\}$, there exist fields $F$, forms $q$ over $F$ and scalars $a\in F^{\x}$ that satisfy the hypotheses of Proposition~\ref{isogroup}.

 

\begin{example}\label{groupiso} Let $L$ a field and $a\in L^{\x}$ be such that $q\perp\< -a\>$ is an anisotropic Pfister neighbour, where $\dim q\neq 2^n$ for $n\in\mathbb{N}\cup \{ 0\}$. Letting $F=L(q\perp -aq)$, we have that $a\in H_F(q)$ and that $q\perp\< -a\>$ is anisotropic over $F$, by Theorem~\ref{H95}. As $q\perp\< -a\>$ is a Pfister neighbour of dimension $\neq 2^n+1$ for $n\in\mathbb{N}\cup \{ 0\}$, it follows that $i_1(q\perp\< -a\>)>1$.
\end{example}

In contrast with the above example, letting $q$ be an arbitrary anisotropic form over $F$ of dimension $2^n$ for some $n\in\mathbb{N}$, the proof of the following theorem demonstrates that there does exist an extension $K/F$ such that $q_K$ is an anisotropic group form. Moreover, when possible, we can find an extension $K/F$ such that $q_K$ is an anisotropic group form that is not round.




\begin{theorem}\label{gnr} Let $q$ be an anisotropic form over $F$ of dimension $2^n$ for $n\in\mathbb{N}$.
\begin{enumerate}[$(i)$]
\item There exists an extension $K/F$ such that $q_K$ is an anisotropic group form.
\smallskip
\item If $q$ is similar to a Pfister form over $F$, then $q_K$ is round for every extension $K/F$ such that $q_K$ is a group form.
\smallskip
\item If $q$ is not similar to a Pfister form over $F$, there exists an extension $K/F$ such that $q_K$ is an anisotropic group form that is not round.
\end{enumerate}
\end{theorem}


\begin{proof} $(i)$ In light of Statement $(iii)$, it suffices to prove this statement in the case where $q$ is similar to a Pfister form over $F$. By Lemma~\ref{simone}, we have that $q$ is a Pfister form if and only if it represents one. Thus, we may let $K=F$ in the case where $1\in D_F(q)$. Otherwise, we may let $K=F(q\perp\< -1\>)$, as $q_K$ is anisotropic by Theorem~\ref{H95}.

$(ii)$ Let $K/F$ be such that $q_K$ is a group form. As $1\in D_K(q)$, we have that $q_K$ is a Pfister form by Lemma~\ref{simone}, whereby $q_K$ is round.

$(iii)$ If $1\notin D_F(q)$, we may consider $q$ as a form over $L=F(q\perp\< -1\>)$, whereby $1\in D_L(q)$. In this case, $q_L$ remains anisotropic by Theorem~\ref{H95}. As $q$ is not similar to a Pfister form over $F$, it follows that $q$ is not hyperbolic over $F(q)$ by \cite[Theorem X.$4.14$]{LAM}. Since $i(q_{L(q)})=i(q_{F(q)(q\perp\< -1\>)})$, we may invoke the Cassels-Pfister Subform Theorem \cite[Theorem X.$4.5$]{LAM} to establish that $q$ is not hyperbolic over $L(q)$, whereby it follows that $q$ is not similar to a Pfister form over $L$ by \cite[Theorem X.$4.14$]{LAM}. Hence, we may assume, without loss of generality, that $1\in D_F(q)$. 



Let $K=F$ if $q$ is a group form over $F$ that is not round. Otherwise, let $L_0=F(\!(x)\!)(q\otimes \< 1,-x\>)$. Since $q$ is not similar to a Pfister form over $F$, it follows that $q$ is not a group form over $L_0$ by Theorem~\ref{transgroup}. Hence, we have that $H_{L_0}(q)\setminus D_{L_0}(q)$ is a non-empty set by Corollary~\ref{group2} (in particular, the proof of Theorem~\ref{transgroup} establishes that $x\in H_{L_0}(q)\setminus D_{L_0}(q)$). Consider the set $$Q(L_0)=\{ q\perp\< -a\>\mid a\in H_{L_0}(q)\setminus D_{L_0}(q)\},$$ which is a non-empty set of anisotropic forms over $L_0$. For $i\geqs 0$, we inductively define $L_{i+1}$ to be the compositum of all function fields of forms in $Q(L_i)$. For all $L_i$ and $a\in H_{L_i}(q)\setminus D_{L_i}(q)$, we have that $q$ is anisotropic over $L_i(q\perp\< -a\>)$ by Theorem~\ref{H95}. Hence, letting $K=\bigcup_{i=0}^{\infty} L_i$, it follows that $q_K$ is anisotropic. Moreover, as $H_{K}(q)= D_{K}(q)$ by construction, it follows that $q_K$ is a group form by Corollary~\ref{group2}. 


It remains to show that $q_K$ is not round. By construction, we have that $q\perp -xq$ is isotropic over $L_0$, whereby $x\in H_K(q)$. Suppose, for the sake of contradiction, that $q\perp -xq$ is hyperbolic over $K$. Hence, for some $i\in\mathbb{N}\cup \{ 0\}$, there exists an extension $L_i' /L_i$ and $a\in (L_i')^{\x}$ such that $((q\perp -xq)_{L_i'})_{\text{an}}$ is hyperbolic over $L_i'(q\perp\< -a\>)$. As a consequence of Elman and Lam's representation theorem \cite[Theorem 1.4]{EL}, there exists a form $p$ over $L_i'$ such that $\dim p<\dim q$ and $((q\perp -xq)_{L_i'})_{\text{an}}\simeq p\perp -xp$ (see \cite[Lemma 3.1]{H4}). Hence, invoking \cite[Corollary 4.2]{KM}, it follows that $$\dim(p\perp -xp)-i\left((p\perp -xp)_{L_i'(q\perp\< -a\>)}\right)\geqs \dim (q\perp\< -a\>)-i_1(q\perp\< -a\>).$$ 
As $i_1(q\perp\< -a\>)=1$ by Theorem~\ref{H95} and \cite[Exercise I.$16$]{LAM}, it follows that $$\dim(p\perp -xp)-\frac {\dim(p\perp -xp)} 2\geqs \dim q,$$ in contradiction to the fact that $\dim p<\dim q$. Hence, having obtained our desired contradiction, we may conclude that $x\notin G_K(q)$, whereby $q_K$ is not round by Corollary~\ref{roundchar}.
\end{proof}

\begin{remark}\label{scully} Scully \cite[Main Theorem]{S2} recently established that, for $p$ and $q$ anisotropic forms over $F$ of dimension at least two with $2^i<\dim q\leqs 2^{i+1}$, if $p_{F(q)}$ is hyperbolic, then $\dim p=2^{i+1}k$ for some $k\in\mathbb{N}$. One may invoke this result to shorten the final component of the above proof.
\end{remark}


Per Example~\ref{groupiso}, in order to answer Question~\ref{groupnotround} in the case where $\dim q\neq 2^n$ for $n\in\mathbb{N}$, we require some additional assumptions regarding the form $q$ over $F$. Orderings are a useful tool in this regard, with their behaviour with respect to function-field extensions being governed by the following result due to Elman, Lam and Wadsworth \cite[Theorem 3.5]{ELW} and, independently, Knebusch \cite[Lemma 10]{GS}.
\begin{theorem}\label{ELW} Let $q$ be a form of dimension at least two over a real field $F$. An ordering $P$ of $F$ extends to $F(q)$ if and only if $q$ is indefinite at $P$.
\end{theorem}

Invoking Theorem~\ref{ELW}, we can resolve Question~\ref{groupnotround} in the case where $q$ is a positive-definite form over a real field. 



\begin{prop}\label{gnrreal} Let $F$ be a real field. Let $q$ be a form over $F$ that is positive definite with respect to some ordering of $F$. If $q$ is not similar to a Pfister form over $F$, there exists an extension $K/F$ such that $q_K$ is an anisotropic group form that is not round.
\end{prop}

\begin{proof} Let $P$ be an ordering of $F$ such that $q$ is positive definite with respect to $P$. If $1\notin D_F(q)$, we may consider $q$ as a form over $L=F(q\perp\< -1\>)$, whereby $P$ is an ordering of $L$ by Theorem~\ref{ELW} and $1\in D_L(q)$. Per the proof of Proposition~\ref{gnr}~$(iii)$, $q$ is not hyperbolic over $L(q)$, and thus is not similar to a Pfister form over $L$. Hence, we may assume, without loss of generality, that $1\in D_F(q)$. 

Let $K=F$ if $q$ is a group form over $F$ that is not round. Otherwise, let $L_0=F(\!(x)\!)(q\otimes \< 1,-x\>)$. Since $q$ is not similar to a Pfister form over $F$, it follows that $q$ is not a group form over $L_0$ by Theorem~\ref{transgroup}. Hence, we have that $H_{L_0}(q)\setminus D_{L_0}(q)$ is a non-empty set by Corollary~\ref{group2}. Moreover, by Theorem~\ref{ELW}, there exist orderings of $L_0$ such that $q$ is positive definite. 

Let $L/L_0$ be an extension such that $q$ is positive definite with respect to an ordering $P$ of $L$ and $H_{L}(q)\setminus D_{L}(q)$ is not empty. Let $a\in H_{L}(q)\setminus D_{L}(q)$, whereby $a\in P$ and the form $q\perp\< -a\>$ has signature $\dim q -1$ with respect to $P$. As $P$ extends to $L(q\perp\< -a\>)$, by Theorem~\ref{ELW}, it thus follows that $i_1(q\perp\< -a\>)=1$. Hence, $q$ is anisotropic over $L(q\perp\< -a\>)$ by Theorem~\ref{km}~$(i)$.


Equipped with the above, we may now proceed with the argument in the proof of Proposition~\ref{gnr}~$(iii)$ to establish the existence of an extension $K/L_0$ such that $q_K$ is an anisotropic group form with $x\in D_{K}(q)\setminus G_{K}(q)$.
\end{proof}

As discussed in \cite{OS2}, many properties of a form $q$ over $F$ are shared by its generic Pfister multiple $q\otimes\< 1,x\>$ over $F(\!(x)\!)$. Invoking Proposition~\ref{group1}, we may show that this is also the case with respect to the group property.


\begin{prop}\label{grouptrans} A form $q$ is a group form over $F$ if and only if $q\otimes\< 1,x\>$ is a group form over $K=F(\!(x)\!)$.
\end{prop}

\begin{proof} Without loss of generality, we may assume that $q$ is anisotropic over $F$.



Let $q$ be a group form over $F$. As every non-zero square class in $K$ can be represented by $a$ or $ax$ for some $a\in F^{\x}$, we first suppose that $q\otimes\< 1,x\>\perp -ax(q\otimes\< 1,x\>)$ is isotropic over $K$ for $a\in F^{\x}$. As $x\in D_K(\< 1,x\>)$, it follows that $q\otimes\< 1,x\>\perp -a(q\otimes\< 1,x\>)$ is isotropic over $K$ in this case. Thus, supposing that $q\otimes\< 1,x\>\perp -a(q\otimes\< 1,x\>)$ is isotropic over $K$ for $a\in F^{\x}$, it suffices to show that $q\otimes\< 1,x\>\perp\<-a\>$ and $q\otimes\< 1,x\>\perp\<-ax\>$ are isotropic over $K$ in accordance with Proposition~\ref{group1}. Invoking Lemma~\ref{Hlemma}, it follows that $q\perp -aq$ is isotropic over $F$, whereby $q\perp\<-a\>$ is isotropic over $F$ by Proposition~\ref{group1}. Hence, it follows that $q\otimes\< 1,x\>\perp\<-a\>$ and $q\otimes\< 1,x\>\perp\<-ax\>$ are isotropic over $K$, as desired.

Conversely, let $q\otimes\< 1,x\>$ be a group form over $K$. Letting $a\in F^{\x}$ be such that $q\perp -aq$ is isotropic over $F$, it follows that $q\otimes\< 1,x\>\perp -a(q\otimes\< 1,x\>)$ is isotropic over $K$, whereby $q\otimes\< 1,x\>\perp -a$ is isotropic over $K$ by Proposition~\ref{group1}. Hence, $q\perp \<-a\>$ is isotropic over $F$ by Lemma~\ref{Hlemma}, whereby $q$ is a group form over $F$ by Proposition~\ref{group1}.
\end{proof}

Combining Proposition~\ref{grouptrans} with Proposition~\ref{laurent} $(ii)$ and \cite[Proposition $3.11$]{OS2}, we obtain the following corollary.

\begin{cor}\label{genericgnr} Let $q$ be an anisotropic group form over $F$ that is not round. Then $q\otimes\< 1,x\>$ is an anisotropic group form over $K=F(\!(x)\!)$ that is not round.
\end{cor}

{\emph{Acknowledgements.} I thank Sylvain Roussey for making his comprehensive PhD thesis available to me. I am grateful to Vincent Astier and Thomas Unger for helpful comments.

\end{document}